\documentclass{amsart}
\usepackage{amssymb}

\usepackage{bm}
\def\<{\langle}
\def\>{\rangle}
\usepackage{upgreek}

\def\E{{\bf E}}

\def\AA{{\mathcal A}}
\def\BB{{\mathcal B}}

\def\DD{{\mathcal D}}

\def\MM{{\mathcal M}}

\def\PP{{\mathcal P}}

\def\SS{{\mathcal S}}
\def\TT{{\mathcal T}}

\def\bbC{\mathbb{C}}

\def\bbN{\mathbb{N}}

\def\bbE{\mathbf{E}}

\def\1{\mathbf{1}}

\newcommand{\toep}{\intercal}

\newcommand{\CD}{{\rm{CD}}}

	\numberwithin{equation}{section}
\newtheorem{theorem}{Theorem}[section]

\newtheorem{corollary}[theorem]{Corollary}
\newtheorem{lemma}[theorem]{Lemma}
\newtheorem{definition}[theorem]{Definition}

\theoremstyle{definition}
\newtheorem{remark}[theorem]{Remark}

\newtheorem{example}[theorem]{Example}
\newtheorem{examples}[theorem]{Examples}

\begin{document}
\title{Algebras of block Toeplitz matrices with commuting entries}
\author{Muhammad Ahsan Khan}
\address{Abdus Salam School of Mathematical Sciences, GC University, Lahore, Pakistan}
\email{ahsan.khan@sms.edu.pk,smsahsankhan@gmail.com}
\author{Dan Timotin}
\address{Simion Stoilow Institute of Mathematics of the Romanian Academy, Bucharest, Romania}

\email{dan.timotin@imar.ro}

\linespread{1.1}

\begin{abstract}
	The maximal algebras of scalar Toeplitz matrices are known to be formed by generalized circulants. The identification of algebras consisting of block Toeplitz matrices is a harder problem, that has received little attention up to now. We consider the case when the block entries of the matrices belong to a commutative algebra $ \AA $. After obtaining some general results, we classify all the maximal algebras for certain particular cases of $ \AA $.
\end{abstract}
\maketitle

\section{Introduction}

Toeplitz matrices, that is, matrices whose entries are constant along diagonals, represent a well studied class of structured
matrices (see, for instance,~\cite{ grenand-szego, ioh, widom}). In particular,  multiplication properties of Toeplitz matrices have been discussed in~\cite{shalom, zimmerman, gu-patton}. Since the product of two Toeplitz matrices is not necessarily a Toeplitz matrix, it is interesting to investigate subsets of Toeplitz matrices that are closed under multiplication---that is, algebras of Toeplitz matrices. It is not obvious that such nontrivial algebras exist (besides the obvious case of the scalar multiples of the identity); but it turns out that one can identify all such maximal algebras~\cite{shalom}; they are given by the so-called $ k $-circulants~\cite{davis}, a generalization of usual circulants. 

The related area of block Toeplitz matrices (see~\cite{botsil, GGK}) is less studied. 
The determination of all maximal algebras formed by such matrices  has been proposed as an open question in~\cite{shalom}. However, the problem seems to be rather hard, and not much progress has been done since. An attempt to a partial solution has been done in~\cite{MAK}, where a certain class of maximal algebras has been determined. We pursue here this investigation, obtaining some general structure properties of such algebras, as well as classifications in  particular cases.

The plan of the paper is the following. After fixing notation and some preliminaries, we discuss in section~3  products of block Toeplitz matrices with commuting entries. Then in the sequel of the paper we consider algebras of block Toeplitz matrices with entries belonging to a fixed maximal commutative algebra of matrices.  In Section~4 we obtain several general properties of such maximal algebras. Various concrete algebras are investigated first in Section~5, and then, more relevantly, in sections 6 and 7, where a complete classification of all maximal block Toeplitz matrices with entries in a singly generated algebra is obtained.

\section{Notations and preliminaries}

We denote by $ \MM_{n,d} $ the space of $ n\times n $ block matrices whose entries are matrices of order $d$; we also write $ \MM_{n}=\MM_{n,1}  $. We prefer to label the indices from $ 0 $ to $ n-1 $; so $T\in \MM_{n,d} $ is written $ T= (T_{i,j})_{i,j=0}^{n-1}$, with $ T_{i,j}\in\MM_d $. Then
$\mathcal{T}_{n,d} \subset \MM_{n,d} $ is the space of block Toeplitz matrices $A = (A_{i-j})_{i,j=0}^{n-1}$. 

If $ \AA $ is a collection of $ d\times d $ matrices, we will use the following notations:
\begin{itemize}
	\item $ \MM_{n,d}[\AA] $ is the collection of $ n\times n $ block matrices whose entries all belong to $ \AA $;
	\item $ \TT_{n,d}[\AA] $ is the collection of $ n\times n $ block Toeplitz matrices whose entries all belong to $ \AA $;
	
	\item $ \DD_{n,d}[\AA] $ is the collection of $ n\times n $ diagonal block Toeplitz matrices whose entries all belong to $ \AA $.
\end{itemize}
Obviously $ \DD_{n,d}[\AA]\subset \TT_{n,d}[\AA] \subset \MM_{n,d}[\AA] $.
The notation will be used mainly when $ \AA $ is an algebra.

A matrix $ T= (T_{i,j})_{i,j=0}^{n-1}$ has $ 2n-1 $ diagonals, corresponding to $ i-j=k $ ($ k=-(n-1), \dots, 0, \dots, n-1 $). A block Toeplitz matrix has all its diagonals constant.

In our context an important role is played by a related notion. A \emph{cyclic diagonal of order $ k $} (where $ k=0,1, \dots, n-1 $) is a matrix $ T= (T_{i,j})_{i,j=0}^{n-1}$ with  $ T_{i,j}\not=0 $ only for 
$ i-j=k $ or $ i-j=k-n $.  We denote the set of all  cyclic diagonals of order $k$ by $\CD_k$; also, $ \CD_k^\TT=\CD_k\cap\mathcal{T}_{n,d} $
A block Toeplitz cyclic diagonal $ A $ is determined by the two elements $ A_k $ and $ A_{k-n} $ for $ k\ge 1 $ and for the single element $ A_0 $ for $ k=0 $.

The next two lemmas concerning the product of cyclic diagonals are proved by direct computations.

\begin{lemma}\label{le:prod gen perm}
	Suppose $ A\in \CD_k $, $ B\in\CD_\ell $. Then:
	\begin{enumerate}
		\item If $ k+\ell\le n-1 $, then $ AB\in \CD_{k+\ell} $.
		
		\item If $ k+\ell\ge n $, then $ AB\in \CD_{k+\ell-n} $.
	\end{enumerate}
\end{lemma}

\begin{lemma}\label{le:prod gen perm toeplitz}
	Suppose $ A=(A_{i-j})\in \CD_k^\TT $, $ B=(B_{i-j})\in \CD_\ell^\TT $, and $ C=AB $. 
	\begin{enumerate}
		\item If $ k+\ell\le n-1 $, then  $ C=(C_{ij}) \in\CD_{k+\ell}$, and its nonzero entries are given by the following formulas:
		\[
		\begin{split}
		C_{j+k+\ell, j}&=A_kB_\ell \text{ if } j+k+\ell\le n-1,\\
		C_{j+k+\ell-n, j}&=A_{k-n}B_\ell \text{ if } j+k+\ell\ge n,\ j+\ell\le n-1,\\
		C_{j+k+\ell-n, j}&=A_{k}B_{\ell-n} \text{ if } \ j+\ell\ge n.
		\end{split}
		\]
		\item If $ k+\ell\ge n $, then  $ C=(C_{ij}) \in\CD_{k+\ell-n}$, and its nonzero entries are given by the following formulas:
		\[
		\begin{split}
		C_{j+k+\ell-n, j}&=A_{k-n}B_\ell \text{ if }  j+\ell\le n-1,\\
		C_{j+k+\ell-n, j}&=A_{k}B_{\ell-n} \text{ if } \ j+\ell\ge n,\ j+k+\ell\le 2n-1,\\
		C_{j+k+\ell-2n, j}&=A_{k-n}B_{\ell-n} \text{ if } \ j+\ell\ge n,\ j+k+\ell\ge 2n.
		\end{split}
		\]
	\end{enumerate}	
\end{lemma}

So the product of two cyclic diagonals is still a cyclic diagonal, while the product of two block Toeplitz cyclic diagonals is in general \emph{not} block Toeplitz.

 In the space $\mathcal{M}_{n,d}$  we define the projections $ \mathbf{E}_k $, $k =
0,\cdots n-1$ as follows: the $(i, j)$ entry of $ \mathbf{E}_k(A)$ is equal to the entry of $A$ if $i-j = k$
or $i-j = k-n$, and is equal to $0$ otherwise. In other words, $ \mathbf{E}_k(A)$ “selects” only
the cyclic diagonal of order $ k $ of a matrix $ A $.  It is clear that any $ \bbE_k $ leaves $ \TT_{n,d} $ invariant.

The group  $ \mathfrak{S}_{n-1} $ of permutations of  $\{1, \dots, n-1\} $ acts on $ \TT_{n,d} $ by permuting cyclic diagonals as follows. Suppose $ \sigma\in\mathfrak{S}_{n-1} $. A cyclic diagonal of order  $ k $ is characterized by the elements $ A_k $ and $ A_{k-n} $. Then  $ B=\sigma  (A) $ is defined by $ B_k=A_{\sigma^{-1}(k)} $, $ B_{k-n}=A_{\sigma^{-1}(k)-n} $. Note that the main diagonal is left unchanged by this action.

In the scalar case $ d=1 $ all maximal algebras of Toeplitz matrices $ T=(t_{i-j})_{i,j=0}^{n-1} $ are identified in~\cite{shalom}. They are characterized by the property that either there exists $ \alpha\in\bbC $ such that $ t_{i-n} =\alpha t_i$ for all $ i=1, \dots, n-1 $, or $ t_i=0 $ for all $ i\ge 1 $ (in the last case one says that $ \alpha=\infty $).  These algebras, that we will denote $ \Pi_\alpha $, appear among the generalizations of circulants in~\cite{davis}, where they are called $ k $-circulants.

Our purpose in this paper is to obtain results of the same type for block Toeplitz matrices. However, the general situation seems too hard to decide. In order to make the problem more tractable,  we will assume in the rest of the paper that the entries of the $ n\times n $ block Toeplitz matrices belong to a fixed commutative maximal algebra $ \AA $ of $ \MM_d $; that is, the block Toeplitz matrices are in $ \TT_{n,d}[\AA] $. Whenever we will speak about maximal algebras of block Toeplitz matrices, this will mean maximal algebras contained in $ \TT_{n,d}[\AA] $.

 There are many types of maximal commutative subalgebras $ \AA $ of $\mathcal{M}_{d}(\mathbb{C})$, and their classification is far from being achieved (see, for instance, \cite{brown1, brown2, brown3}). Here are some examples, to which we will refer in the sequel.
 
 \begin{examples}\label{ex:maximal algebras}
 	\begin{enumerate}
 		\item If we fix a basis in $ \mathbb{C}^d $, then the algebra of diagonal matrices $\mathcal{D}_d$ is maximal commutative.
 		
 		\item 	The generalized circulants $ \Pi_\alpha $ defined above are 
 		maximal commutative subalgebras. Two noteworthy cases are $ \Pi_0 $ (the lower triangular Toeplitz matrices) and $ \Pi_\infty $ (the upper triangular Toeplitz matrices).
 		
 		\item 
 		Fix positive integers $ p, q $, such that $p+q=d$.  Consider the family $\mathcal{O}_{p, q}$ of  matrices that have with respect to the decomposition $\mathbb{C}^{d}=\mathbb{C}^{p}\oplus\mathbb{C}^{q}$ the form
 		\[
 		\begin{pmatrix}
 		\lambda I_{p}& X\\
 		0            &\lambda I_{q}
 		\end{pmatrix}
 		\]
 		with $\lambda\in\mathbb{C}  $ and $ X$ an arbitrary $ p\times q $ matrix with complex entries.
 		It is shown in~\cite{brown1} that, for $ |p-q|\le 1 $, $ 	\mathcal{O}_{p, q} $ 
 		is  a maximal commutative algebra.
 		
 		\item 	Suppose that $M$ is a nonderogatory matrix; that is, its minimal polynomial is equal to its characteristic polynomial. Then the algebra $\mathcal{P}(M)$ generated by $M$ is   maximal commutative.
 		
 	\end{enumerate}
 	
 \end{examples}

\section{Block Toeplitz matrices with commuting entries}

We gather in the next lemma two basic facts about the product of block Toeplitz matrices that can be found in~\cite{shalom}, Proposition 2.1 and Theorem 1.8, respectively (see also~\cite{zimmerman}).

\begin{lemma}\label{le:basic condition for product toeplitz}
	\begin{itemize}
		\item[\rm(i)]  	Suppose $A=(A_{i-j})_{i,j=0}^{n-1}$
		and $B = (B_{i-j})_{i.j=0}^{n-1}$
		are two block Toeplitz matrices. Then $AB$ is a block Toeplitz matrix if and only if
		\begin{equation}\label{prod}
		A_iB_{j-n}=A_{i-n}B_j\quad\hbox{for all}\quad i,j =1,\cdots n-1.
		\end{equation}
		
		\item[\rm(ii)] $ A^k $ is block Toeplitz for all $ k\in\bbN $ if and only if 
		\begin{equation}\label{eq:powers toeplitz}
			A_iA_0^m A_{j-n}=A_{i-n}A_0^m A_j
		\end{equation}	
		for all $i, j=1,\dots, n-1 $ and $ m\in\bbN $.
	\end{itemize}

\end{lemma}

Condition~\eqref{prod} states a relation between the cyclic diagonals of order $ i $ and $ j $ of $ A $ and $ B $, respectively. If two cyclic diagonals satisfy~\eqref{prod}, we will say that they are \emph{compatible}. Thus Lemma~\ref{le:basic condition for product toeplitz} may be restated as saying that $ AB $ is Toeplitz if and only if any cyclic diagonal of nonzero order of $ A $ is compatible with any cyclic diagonal of nonzero order of $ B $. In particular, the condition does not depend on the main diagonals of the two matrices. 

If $ A,B $ are two block Toeplitz matrices, we will use the notation $ A\toep B $ to indicate that the product $ AB $ is also a block Toeplitz matrix. If $ \SS $ is a set of block Toeplitz matrices, then we write $ A\toep\SS $ if $ A\toep B $ for any $ B\in\SS $.

The following corollary gathers some consequences of Lemma~\ref{le:basic condition for product toeplitz}.

\begin{corollary}\label{co:product}
	Suppose $A$ and $B$ are block Toeplitz matrices with $ A\toep B $.
	\begin{itemize}
		\item[\rm(i)] If  $ \sigma, \tau\in \mathfrak{S}_{n-1} $, then $ \sigma(A) \toep \tau(B) $.
		
		\item[\rm(ii)] If the entries of $A$ commute with the entries of $B$, then  $AB = BA$.
		
		\item[\rm(iii)] If $\mathcal{B}$ is an algebra of block Toeplitz matrices with entries in a
		commutative algebra $\mathcal{A}$, then $\mathcal{B}$ is commutative. 
		
			\item[\rm(iv)] $ A\toep \DD_{n,d} $. 
	\end{itemize}
	\end{corollary}

\begin{proof} (i) If any cyclic diagonal of $ A $ is compatible with any cyclic diagonal of $ B $, then the same is true about $ \sigma(A) $ and $ \tau(B) $.

(ii)	 Fix $i,j$. If $C=AB$, $C^\prime=BA$, we have then 
	\begin{align*}
	C_{i,j}
	&= \sum_{k=0}^{n-1}A_{i-k}B_{k-j}\\
	C_{i,j}^\prime
	&= \sum_{k=0}^{n-1}B_{i-k}A_{k-j}
	\end{align*}
	Applying commutativity of entries and rewriting the sum by denoting $t=i+j-k$, we obtain 
	\[
	C_{i.j}^\prime= \sum_{k=0}^{n-1}A_{k-j}B_{i-k}=\sum_{t=i+j-(n-1)}^{i+j}A_{i-t}B_{t-j}
	\]
	If $i+j=n-1$ then the last sum coincides exactly with the formula for $C_i,j$. So $C_{i,j}^\prime=C_{i,j}$. Let us then suppose, for instance, that $i+j<n-1$. Then only part of the sum is the same, and rest yields
	\begin{align*}
	C_{i,j}-C_{i,j}^\prime
	&= \sum_{k=i+j+1}^{n-1}A_{i-k}B_{k-j}- \sum_{i+j-(n-1)}^{-1}A_{i-k}B_{k-j}\\
	&=\sum_{k=i+j+1}^{n-1}(A_{i-k}B_{k-j}-A_{i-k+n}B_{k-j-n})
	\end{align*}
	By Lemma \ref{le:basic condition for product toeplitz} (i) each of the terms in the sum is zero, and so $C_{i,j}-C_{i,j}^\prime=0$. A similar argument works for $i+j>n-1$.
	
	(iii) is an immediate consequence of (ii).
	
	(iv) follows from Lemma~\ref{le:basic condition for product toeplitz} (i).
\end{proof}

In the scalar case $ \AA=\bbC $, Corollary~\ref{co:product} (iii) recaptures the commutativity of the generalized circulants $ \Pi_\alpha $.

\section{General algebras of block Toeplitz matrices}\label{se:diagonals}

We start with the following simple observation.
\begin{lemma}\label{l1}
	If $ A,B $ are  block Toeplitz matrices such that   $A\toep B$, then $ \mathbf{E}_i(A)\toep \mathbf{E}_j(B)$ for any $i,j = 0,\dots, n-1$.
\end{lemma}
\begin{proof}
	The proof follows immediately from Lemma~\ref{le:basic condition for product toeplitz} (i). If any cyclic diagonal of nonzero order of $ A $ is compatible with any cyclic diagonal of nonzero order of $ B $, it is obvious that the same is true for  $ \mathbf{E}_i(A)$ and $ \mathbf{E}_j(B)$. 
\end{proof}

\begin{lemma}\label{le:cyclic diagonal power}
	Suppose $ D$ is a block Toeplitz cyclic diagonal with commuting entries. Then $ D^k $ are also block Toeplitz cyclic diagonals for any $ k\in\bbN $.
\end{lemma}

\begin{proof} Suppose $ D\in  \CD_p^\TT $. The result is immediate for $ p=0 $, so we shall assume that $ p\ge 1 $.
The fact that $ D^k $ is a cyclic diagonal follows from Lemma~\ref{le:prod gen perm}. To prove that it is block Toeplitz, we check conditions~\eqref{eq:powers toeplitz} in Lemma~\ref{le:basic condition for product toeplitz} (ii). If $ m\ge0 $ we have $ D_0^m=0 $, and so both sides are~0. For $ m=0 $ both sides of~\eqref{eq:powers toeplitz} are again 0, except when $ i=j=p $. Then the equality follows from the commutation of $ D_p $ and~$ D_{p-n} $.
\end{proof}

\begin{lemma}\label{le:product of cyclic diagonals toeplitz}
	Suppose $ E,F $ are block Toeplitz cyclic diagonals, and the entries of $ E $ commute with those of $ F $. If $ A $ is a block Toeplitz matrix such that $ A\toep E $ and $ A\toep F $, then $ A\toep (EF) $.
\end{lemma}

\begin{proof} Formulas~\eqref{prod} are true for the pairs $ A,E $ and for $ A,F $. In order to show that $ A\toep (EF) $, we have to check them 
 for $ A $ and $ C=EF $. One has to consider different cases, according to the formulas in Lemma~\ref{le:prod gen perm toeplitz}. Here is one case (we leave the remaining to the reader). Suppose $ E\in \CD_k^\TT $, $ F\in\CD_\ell^\TT $, and $ k+\ell\le n-1 $. We have then $ C\in \CD_{k+\ell}^\TT $, and 
	\[
	\begin{split}
	C_{k+\ell}&=E_kF_\ell\\
	C_{k+\ell-n}&= \text{either }E_{k-n}F_\ell\text{ or } E_kF_{\ell-n}.
	\end{split}
	\]
	Since $ C $ is a cyclic diagonal of order $ k+\ell $, the only relations~\eqref{prod} that we have to check are
	\[
	A_iC_{k+\ell-n}= A_{i-n}C_{k+\ell}
	\]
	for $ i=1, \dots, n-1 $.
	This becomes in the first case
	\[
	A_iE_{k-n}F_\ell=A_{i-n}E_kF_\ell
	\]
	and in the second
	\[
	A_iE_{k}F_{\ell-n}=A_{i-n}E_kF_\ell.
	\]
	The first follows from~\eqref{prod} for $ A, E$, while the second follows from~\eqref{prod} for $ A, F$ (since by hypothesis $E_{k}F_{\ell-n}=F_{\ell-n} E_{k} $ and $ E_kF_\ell= F_\ell E_k $).
	\end{proof}

\begin{corollary}\label{co:adding a cyclic diagonal}
	Let $ \AA $ be a maximal commutative algebra of $ d\times d $ matrices. Suppose $ \BB $ is a commutative algebra of block Toeplitz matrices with entries in  $ \AA $, and $ D $ is a block Toeplitz cyclic diagonal with entries in $ \AA $, such that $ D\toep\BB $. Then algebra $ \BB_1 $ generated by $ \BB $ and $ D $ is formed by block Toeplitz matrices.
\end{corollary}

\begin{proof}  Lemma~\ref{le:cyclic diagonal power} implies that $ D^k $ is a block Toeplitz cyclic diagonal. Using recurrence, it follows from Lemma~\ref{le:product of cyclic diagonals toeplitz} that $ D^k\toep \BB $ for all $ k $, and therefore  $ D^k $ is also in the commutant of $ \BB $ by Corollary~\ref{co:product} (ii).
The algebra $ \BB_1 $ is then formed by polynomials in $ D $ with coefficients in $ \BB $, and any monomial of the type $ BD^k $ is block Toeplitz.  It follows that $ \BB_1 $ is formed by block Toeplitz matrices.
\end{proof}

The next theorem is our main result about general maximal algebras contained in  $ \TT_{n,d}[\AA] $, with $ \AA $  a maximal commutative algebra of $ d\times d  $ matrices.

\begin{theorem}\label{th:general facts about maximal algebras}
	Suppose $\mathcal{B}$ is a maximal subalgebra of $ \TT_{n,d}[\AA] $.
	\begin{itemize}
		\item[\rm(i)]  $\DD_{n,d}[\AA]\subset \BB $.

		\item[\rm (ii)]  If $A\in\mathcal{B}$, then $ \mathbf{E}_k (A)\in\mathcal{B}$ for all $0\leq k\leq n-1$.
		
		\item[(iii)] $ \BB $ coincides with the vector space generated by $ \E_k(\BB) $, for $ k=0,\dots, n-1 $.
				
					\item[\rm (iv)]  $ \BB $ is invariant with respect to the action of  $ \mathfrak{S}_{n-1} $.
	
		\item [\rm (v)] $ \mathbf{E}_k(\mathcal{B})\not=\{0\} $ for any $ k=0,\dots, n-1 $.

	\end{itemize}
	
\end{theorem}

\begin{proof} 
	
	(i) 
	If $ D $ is a diagonal block Toeplitz matrix, it follows from Corollary~\ref{co:product} (iv) that $ D\toep \BB $. If $ D\notin\BB $, we may apply Corollary~\ref{co:adding a cyclic diagonal} and obtain the algebra $ \BB_1\supset\BB $ composed of block Toeplitz matrices. But maximality implies $ \BB_1=\BB $, which proves~(i).

Suppose $A\in\mathcal{B}$, and fix $ k $. It is immediate from~\eqref{prod} that $ \mathbf{E}_k (A)\toep\mathcal{B}$; so Corollary~\ref{co:adding a cyclic diagonal} and maximality of $ \BB $ imply that $ \mathbf{E}_k (A)\in\mathcal{B}$. This proves (ii); then (iii) is an immediate consequence of (ii).

	To prove (iv), it is enough, by (iii), to show that if $ D $ is a cyclic diagonal in $ \BB $ and  $ \sigma\in \mathfrak{S}_{n-1} $, then $ \sigma(D)\in \BB $. 
	For any $ A\in\BB $ we have $ D\toep A $; therefore, applying  Corollary~\ref{co:product} (i)  to $ \sigma $ and $ \tau= $identity, it follows that $ \sigma(D)\toep A $. So $  \sigma(D)\toep \BB$, and we may apply Corollary~\ref{co:adding a cyclic diagonal} to $ \sigma(D) $ and $ \BB $. The algebra $ \BB_1 $ obtained contains $ \sigma(A) $ and must be equal to $ \BB $ by maximality.

	For (v), let us first note that $ \BB $ cannot be equal to $\DD_{n,d}[\AA]$; indeed, the algebra generated by $\DD_{n,d}[\AA]$ and any cyclic diagonal is formed by block Toeplitz matrices by Corollary~\ref{co:adding a cyclic diagonal}  and it contains $\DD_{n,d}[\AA]$ strictly. Therefore	
	 $ \mathbf{E}_k(\mathcal{B})\not=\{0\} $ for at least one value of $ k \ge 1 $. Applying (iv) it follows that 
	 $ \mathbf{E}_k(\mathcal{B})\not=\{0\} $ for all $ k\ge 1 $.
\end{proof}

The main part of Theorem~\ref{th:general facts about maximal algebras} can be summarized as follows. First, a maximal subalgebra of $ \TT_{n,d}[\AA] $  is generated as a linear subspace by the cyclic diagonals it contains. Secondly, a cyclic diagonal being characterized by two matrices $ A_k $ and $ A_{k-n} $, the space of all pairs of such matrices are the same  for different values of~$ k $.

\medskip
We end this section by recalling some results of~\cite{MAK}. Suppose $ \AA $ is a maximal commutative algebra of $ \MM_d $. Let us then fix $ A,B\in \AA $, with the property that
\begin{equation}\label{eq:kerAkerB}
\ker A\cap \ker B=\{0\}
\end{equation}
We define the family $\mathcal{F}_{A,B}^{\mathcal{A}}$ by
\begin{equation}\label{eq:definition of FAB}
\mathcal{F}_{A,B}^{\mathcal{A}}=
\left\lbrace 
(T_{p-q})_{p,q=0}^{n-1}
:
T_{j}\in\mathcal{A},\quad AT_{j}=BT_{j-n}, j=1,2,\cdots n-1
\right\rbrace .
\end{equation}

The main result from~\cite{MAK} may be summarized as follows.

\begin{theorem}\label{th:MAK1}
	\begin{itemize}
		\item[(i)] The families $\mathcal{F}_{A,B}^{\mathcal{A}}$ are maximal subalgebras.
		
		\item[(ii)] We have $ \mathcal{F}_{A,B}^{\mathcal{A}} =\mathcal{F}_{A',B'}^{\mathcal{A}} $ if and only if $ AB'=A'B $.
		
			\item[(iii)] Suppose $\mathcal{B}$ contains an element  $T=(T_{i-j})_{i,j=0}^{n-1}$ such that $T_{r}$ is invertible for some $r\neq0$. Then $\mathcal{B}=\mathcal{F}_{A,B}^{\mathcal{A}}$ for some $A,B\in \AA$ satisfying~\eqref{eq:kerAkerB}.  
		
	\end{itemize}	
\end{theorem}

The only point not stated explicitely in~\cite{MAK} is (ii), for which we sketch here
 a short proof. 	Suppose that $AB^{\prime}=BA^{\prime}$, and take $T=(T_{i-j})_{i,j=0}^{n-1}\in\mathcal{F}_{A,B}^{\mathcal{A}}$. If we multiply the defining relation $ 	AT_{j}=BT_{j-n} $ with $ A' $, we obtain, using $AB^{\prime}=BA^{\prime}$,
	\begin{equation}\label{eq:l1}
A(A^{\prime}T_{j}-B^{\prime}T_{j-n})=0.
\end{equation} 
Similarly, multiplication with $ B' $ leads to
	\begin{equation}\label{eq:l2}
B(A^{\prime}T_{j}-B^{\prime}T_{j-n})=0.
\end{equation} 
From \eqref{eq:l1} and \eqref{eq:l2} it follows easily, using~\eqref{eq:kerAkerB}, that $A^{\prime}T_{j}=B^{\prime}T_{j-n}$, and so  $T\in\mathcal{F}_{A^{\prime},B^{\prime}}^{\mathcal{A}}$. The reverse inclusion is proved in the same manner.

Conversely suppose that $\mathcal{F}_{A,B}^{\mathcal{A}}=\mathcal{F}_{A^{\prime},B^{\prime}}^{\mathcal{A}}$. The matrix $T=(T_{i-j})_{i,j=0}^{n-1} $ defined by $T_{j}=B$ and $T_{j-n}=A$ is obviously in $ \mathcal{F}_{A,B}^{\mathcal{A}} $. It must then be in $ \mathcal{F}_{A',B'}^{\mathcal{A}} $, whence $ AB'=A'B $.

\medskip

We will call an algebra of type $ \mathcal{F}_{A,B}^{\mathcal{A}} $, where $ A,B $ satisfy condition~\eqref{eq:kerAkerB}, a \emph{pseudocirculant} algebra. The name is suggested by the fact that these algebras generalize the $ k $-circulant algebras $ \Pi_\alpha $ from the scalar case.

Theorem~\ref{th:MAK1} states thus, in particular, that any maximal algebra that contains at least one element with an invertible off diagonal entry is a pseudocirculant algebra. As we will see below, in general these are  not all the maximal algebras.

\section{Concrete cases: two examples}

Beyond the general results stated in the previous section, the determination of all maximal algebras depends of course on the particular algebra of coefficients $ \AA $. In this section we will present two simple cases that were discussed in ~\cite{MAK}, in which 
all or almost all maximal algebras are pseudocirculant.
 The rest of the paper will be devoted to the case when $\AA$ is singly generated, which turns out to be much more interesting. In all three cases considered we will obtain a complete description of the maximal algebras.

\subsection{}

The case when $ \AA=\DD_d $ is the algebra of $ d\times d $ diagonal matrices is discussed in~\cite[Section 5]{MAK}, where it is proved that~\eqref{eq:definition of FAB} describes all maximal subalgebras  of $ \TT_{n,d} $. 

\begin{theorem}\label{th:MAK2}
	 If $ \AA=\DD_d $, then any maximal algebra contained in $ \TT_{n,d}[\DD_d]  $ is a pseudocirculant algebra. 	
\end{theorem}
Section~5 of~\cite{MAK} actually says more, giving a precise description of maximal subalgebras. After a reshuffling of bases, it is easy to see that $ \TT_{n,d}[\DD_d] $ is isomorphic to the direct sum of $ d $ copies of $ \TT_{n,1} $. Any maximal algebra splits with respect of this decomposition, and is therefore a direct sum of the generalized circulants $ \Pi_\alpha $ defined in the preliminaries. So the maximal algebras of block Toeplitz matrices are in one-to-one correspondence with the $ d $-tuples $ (\alpha_1, \dots, \alpha_d) $, with $ \alpha_i\in\bbC\cup \{\infty\} $.

\subsection{}

The next  example is taken from~\cite[Section 6]{MAK}. 
Fix positive integers $p,q$, such that $p+q=d$ and $ |p-q|\le 1 $.  Consider the family of  matrices $\mathcal{O}_{p, q}$ defined in Example~\ref{ex:maximal algebras} (3) above. The next result is proved in~\cite[Section 6]{MAK}.

\begin{theorem}\label{th:schur}
		The set $\mathcal{S}$ of  block Toeplitz matrices 
		in  $ \TT_{n,d}[\mathcal{O}_{p, q}] $
	 which have all off diagonal entries noninvertible forms a maximal subalgebra of $ \TT_{n,d}[\mathcal{O}_{p,q}] $ that is not a pseudocirculant algebra.
\end{theorem}

It follows that if $ \AA=\mathcal{O}_{p,q} $ there is a single maximal subalgebra, namely $ \SS $, that is not  pseudocirculant.

The next sections will determine all maximal subalgebras of $ \TT_{n,d}[\PP(M)] $, where $ M $ is a nonderogatory matrix. It will be seen in Example~\ref{ex:k+k-} that in that case we have many different subalgebras all of whose nondiagonal entries are noninvertible.

\section{Block Toeplitz matrices with  entries in a singly generated algebra}

Let $M$ be a fixed nonderogatory $d\times d$ matrix with  minimal (and characteristic) polynomial  $\mathbf p(X)$. In the sequel we will use bold letters to denote polynomials; $ \bm 1 $~will be the unit in the ring of polynomials $ \bbC[X] $.

Our purpose is to obtain a complete description of all maximal subalgebras of $ \TT_{n,d}[\PP(M)] $, where $ \PP(M) $ is the algebra generated by $ M $.

\begin{definition}\label{de:algebras A(r, rho)}
	Suppose $\mathbf p_+, \mathbf p_-,\bm \upchi\in \bbC[X]$ are three polynomials, 
	such that $ \mathbf p_+\mathbf p_-$ divides $\mathbf p $, while $\bm \upchi $ and $\mathbf p $ are relatively prime; denote $ \mathbf q=\mathbf p/\mathbf p_+\mathbf p_- $.
 We define  $ \mathfrak{B}(\mathbf p_+,\mathbf p_-,\bm  \upchi) $ to be the set of matrices $ A=(A_{i-j})_{i,j=0}^{n-1} $, where:
	\begin{itemize}
		\item[(a)] $A_i\in\PP(M)$ for all $ i $ with $ -(n-1)\le i\le n-1 $;
		
		\item[(b)] $A_i=\mathbf p_+(M)\mathbf a_i(M)$, $ A_{i-n}=\mathbf p_-(M)\mathbf a_{i-n}(M) $ for $ i\ge 1 $, where $\mathbf a_i\in\bbC[X] $;
		
		\item[(c)]  $\mathbf q$ divides $\mathbf a_{i}-\bm  \upchi\mathbf  a_{i-n}$ for all $ i\ge 1 $.
	\end{itemize}

	\end{definition}

\begin{theorem}\label{th:algebras A(r, rho)}

	 Any $ \mathfrak{B}(\mathbf p_+,\mathbf p_-,\bm  \upchi) $ is a commutative algebra of block Toeplitz matrices. 
			
\end{theorem}

\begin{proof} (i)
		It is easy to check that $ \mathfrak{B}(\mathbf p_+,\mathbf p_-,\bm  \upchi) $ is a vector space that is linearly generated by its subspaces $\mathfrak{B}_k:=\mathfrak{B}(\mathbf p_+,\mathbf  p_-,\bm  \upchi)\cap \CD_k$, $k=0, \dots, n-1$. 
	
	Fix then $k,\ell$, and take $A\in \mathfrak{B}_k$, $B\in \mathfrak{B}_\ell$. Therefore, if
	$A=(A_{i-j})_{i,j=1}^{n-1}$, $B=(B_{i-j})_{i,j=1}^{n-1}$, all entries of $A,B$ are zero except $A_k=\mathbf p_+(M)\mathbf a_k(M)$, $A_{k-n}=\mathbf p_-(M)\mathbf a_{k-n}(M)$, $B_\ell= \mathbf p_+(M)\mathbf b_\ell(M)$, $B_{\ell-n}=\mathbf p_-(M)\mathbf b_{\ell-n}(M)$, and there are 
	polynomials $\mathbf u,\mathbf v\in\mathbb{C}[X]$ such that 
	\begin{equation} \label{p1}
	\begin{split}
		\mathbf a_k-\bm  \upchi \mathbf a_{k-n}= \mathbf q\mathbf u , \\
	\mathbf b_\ell-\bm  \upchi\mathbf b_{\ell-n}= \mathbf q\mathbf v.
	\end{split}
	\end{equation}
	
We want to use Lemma~\ref{le:basic condition for product toeplitz} to show that $C=AB$ is also a block Toeplitz matrix. Formula~\eqref{prod} has to be checked only for $p=k$, $q=\ell$, since for other values of $p,q$ both terms are zero. It becomes then  
\[
\mathbf p_+(M)\mathbf p_-(M)\mathbf a_k(M)\mathbf b_{\ell-n}(M)= \mathbf p_+(M)\mathbf p_-(M)\mathbf  a_{k-n}(M)\mathbf b_\ell(M).
\]
Using \eqref{p1}, we have
\[
\begin{split}
&\mathbf p_+(M)\mathbf p_-(M)\mathbf a_k(M)\mathbf b_{\ell-n}(M)- \mathbf p_+(M)\mathbf p_-(M) \mathbf a_{k-n}(M)\mathbf b_\ell(M)
\\
&\qquad =\mathbf p_+(M) \mathbf p_-(M)
[\bm  \upchi(M)\mathbf a_{k-n}(M)+ \mathbf q(M)\mathbf u(M) ]\mathbf b_{\ell-n}(M)\\
&\qquad\qquad\qquad-\mathbf p_+(M)\mathbf p_-(M)
\mathbf a_{k-n}(M)
[\bm  \upchi(M)\mathbf b_{\ell-n}(M)+\mathbf q(M)\mathbf v(M)]\\
&\qquad=\mathbf p(M)[\mathbf u(M) \mathbf b_{\ell-n}(M)-\mathbf a_{k-n}(M)\mathbf v(M)]=0.\\
\end{split}
\]
Therefore Lemma~\ref{le:basic condition for product toeplitz} implies that $C=AB$ is  a block Toeplitz matrix. 	
	
	We have now to  prove that $ C=(C_{i,j})_{i,j=0}^{n-1} $ satisfies the condition in the definition of $ \mathfrak{B}(\mathbf p_+, \mathbf p_-,\bm  \upchi) $. As a consequence of Lemma~\ref{le:prod gen perm}, we have to consider two cases.
	
	The first case is $k+\ell\le n-1$. Then $C\in \CD_{k+\ell}$ by Lemma~\ref{le:prod gen perm}. We already know that $ C $ is a block Toeplitz matrix, so it is enough to take check $ C_{k+\ell, 0} $ and $ C_{k+\ell-n, 0}  $. Also, we will assume  $ k,\ell\ge1 $; if one of them is 0 the formulas are simpler.
	
 By Lemma~\ref{le:prod gen perm toeplitz} we have
	\[
	\begin{split}
	C_{k+\ell, 0}&=\mathbf p_+(M)^2\mathbf a_k(M)\mathbf b_\ell(M), \\
	C_{k+\ell-n, 0} &=
	\mathbf p_+(M)\mathbf p_-(M)\mathbf a_k(M)\mathbf b_{\ell-n}(M).
	\end{split}
	\]
	So 
	\[
	C_{k+\ell, 0}= \mathbf p_+(M)\mathbf c_{k+\ell}(M), \qquad
	C_{k+\ell-n, 0}= \mathbf p_-(M)\mathbf c_{k+\ell-n}(M),
	\]
	where we may take
	\[
	\mathbf c_{k+\ell}= \mathbf p_+\mathbf a_k\mathbf b_\ell,\quad
	\mathbf c_{k+\ell-n}= \mathbf p_+\mathbf a_k\mathbf b_{\ell-n}.
	\]
	Then 
	\[
	\mathbf c_{k+\ell}- \bm  \upchi\mathbf c_{k+\ell-n} 
	= \mathbf a_k \mathbf p_+ (\mathbf b_\ell - \bm  \upchi \mathbf b_{\ell-n}  ).
	\]
Since $\mathbf  q  $ divides  $\mathbf b_\ell - \bm  \upchi \mathbf b_{\ell-n}    $, it also divides  $ \mathbf c_{k+\ell} - \bm  \upchi \mathbf c_{k+\ell-n}  $, which is what we wanted to prove.

We leave to the reader to verify the second case $k+\ell\ge n$.
We have thus shown that if $ A $ and $ B $ belong to two of the subspaces $ \mathfrak{B}_k $, then $ C=AB $ also belongs to one of these subspaces. This fact extends by linearity to the whole of $\mathfrak{B}({\mathbf p_+, \mathbf p_-, \bm  \upchi})$, which is therefore an algebra. Finally, it is a commutative algebra as a consequence of Corollary~\ref{co:product} (iii).
\end{proof}

The following result shows, in particular, that any two algebras of this type are either equal or not comparable.

\begin{theorem}\label{th:uniqueness}
	Suppose $ \mathbf p_+,\mathbf p_-, \bm \upchi,  \mathbf p'_+,\mathbf p'_-,\bm  \upchi'\in\bbC[X] $ are such that $  \mathbf p_+\mathbf p_- $
	and $  \mathbf p'_+\mathbf p'_- $ divide $ \mathbf p $,  $\bm  \upchi $ and $  \mathbf p $ are relatively prime, and $\bm  \upchi' $ and $  \mathbf p $ are also relatively prime. The following assertions are equivalent:
	\begin{enumerate}
		\item $\mathfrak{B}(\mathbf p_+,\mathbf p_-, \bm  \upchi)=\mathfrak{B}(\mathbf p_+^{\prime},\mathbf p_-^{\prime}, \bm  \upchi^{\prime})$.
		
			\item $\mathfrak{B}(\mathbf p_+,\mathbf p_-, \bm  \upchi)\subset\mathfrak{B}(\mathbf p_+^{\prime},\mathbf p_-^{\prime}, \bm  \upchi^{\prime})$.
			
			\item We have 
			 \begin{itemize}
				\item[(i)]  $\mathbf p_+^{\prime}$ is a nonzero scalar multiple of $ \mathbf p_+ $,
				\item[(ii)] $\mathbf p_-^{\prime}$ is a nonzero scalar multiple of $ \mathbf p_- $,
				\item[(iii)] $ \mathbf q $ divides $ \bm  \upchi - \bm  \upchi^{\prime} $,
			
			\end{itemize}

	\end{enumerate}

\end{theorem}

\begin{proof}

The implication (1)$ \implies $(2) is obvious, while (3)$ \implies $(1) is easy to check. So we are left to prove that (2)$ \implies $(3).
 
 Suppose then that  $\mathfrak{B}(\mathbf p_+,\mathbf p_-, \bm  \upchi)\subset\mathfrak{B}(\mathbf p_+',\mathbf p_-', \bm  \upchi')$.
Define $ E=(e_{i-j})\in \CD^\TT_1 $ by 
\begin{equation}\label{eq:def of E}
e_i=
\begin{cases}
\mathbf p_+(M) \bm  \upchi(M)&\text{for }i=1,\\
 \mathbf p_-(M) &\text{for }i=1-n,\\
 0&\text{otherwise.}
\end{cases}
\end{equation} 
It is immediate that $E\in \mathfrak{B}(\mathbf p_+,\mathbf p_-, \bm  \upchi)  $, whence $E\in \mathfrak{B}(\mathbf p'_+,\mathbf p'_-, \bm  \upchi')  $. Therefore $ e_1=\mathbf p'_+(M)\mathbf a'_1(M) $, whence $ \mathbf p_+  \bm  \upchi -\mathbf p'_+ \mathbf a'_1  $ is a multiple of $ \mathbf p  $. So $ \mathbf p_+ \bm  \upchi $ is divisible by $ \mathbf p'_+ $. But $  \bm  \upchi $ and $\mathbf  p $ are relatively prime, whence $  \bm  \upchi $ and $ \mathbf p'_+ $ are also relatively prime. It follows that $ \mathbf p_+ $is divisible by $ \mathbf p'_+ $. The case of $ \mathbf p_- $ and $ \mathbf p'_- $ is similar.

Let us now fix  $ \alpha, \beta\in\bbC $ and
define  $ J(\alpha,\beta)\in \CD_1^\TT $ by 
\[
J(\alpha,\beta)_{i,j}=
\begin{cases}
\alpha \mathbf q(M)\mathbf p_+(M) \bm  \upchi(M)&\text{for }i-j=1,\\
 \beta \mathbf q(M) \mathbf p_-(M) &\text{for }i-j=1-n,\\
0&\text{otherwise}.
\end{cases}
\]
It is immediate that $ J(\alpha, \beta)\in \mathfrak{B}(\mathbf p_+,\mathbf p_-, \bm  \upchi) $, so it also belongs to
 $  \mathfrak{B}(\mathbf p_+',\mathbf p'_-, \bm  \upchi') $. By the definition of this latter algebra, there exist polynomials $ \mathbf a'_1 , \mathbf a'_{1-n}  $
such that $ \alpha \mathbf q(M)\mathbf p_+(M) \bm  \upchi(M)= \mathbf p'_+(M)\mathbf a'_1(M)$, 
$ \beta q(M) \mathbf p_-(M)= \mathbf p_-(M) \mathbf a'_{1-n}(M) $ and $ \mathbf q $ divides
$ \mathbf a_1' - \bm  \upchi'\mathbf  a_{1-n}' $.
So there exist polynomials $\mathbf  u ,\mathbf  v ,\mathbf  w  $ such that
\begin{align}
\alpha \mathbf q \mathbf p_+  \bm  \upchi &= \mathbf p'_+ \mathbf a'_1 +\mathbf p  \mathbf u ,\label{eq:11}\\
\beta \mathbf q  \mathbf p_- &= \mathbf p_-  \mathbf a'_{1-n} +\mathbf p \mathbf  v ,\label{eq:12}\\
 \mathbf a_1' - \bm  \upchi' \mathbf a_{1-n}' &= \mathbf q' \mathbf w .\label{eq:13}
\end{align}

Multiply~\eqref{eq:11} by $\mathbf p'_- $,~\eqref{eq:11} by $\mathbf p'_+  \bm  \upchi' $ and substracting, we obtain, using~\eqref{eq:13},
\[
\begin{split}
&\alpha \mathbf q \mathbf p_+  \bm  \upchi \mathbf p'_- 
-\beta \mathbf q  \mathbf p_- \mathbf p'_+  \bm  \upchi' \\
&\qquad\qquad=
\mathbf p'_+ \mathbf p'_- (\mathbf  a_1' - \bm  \upchi'\mathbf  a_{1-n}' )\\
&\qquad\qquad\qquad\qquad+\mathbf p \big(\mathbf p'_-\mathbf  u -\mathbf p'_+  \bm  \upchi' \mathbf v \big)\\
&\qquad\qquad= \mathbf p \big( \mathbf w +\mathbf p'_-\mathbf  u -\mathbf p'_+  \bm  \upchi' \mathbf v  \big).
\end{split}
\]
The right hand side is divisible by $\mathbf  p $, and so the same must be true about the left hand side for any values of $ \alpha, \beta $, which means that both terms have to be divisible by $\mathbf  p $. 

But $ \mathbf q\mathbf p_+\mathbf p'_- \bm  \upchi $ divisible by $\mathbf  p $ implies $ \mathbf q\mathbf p_+\mathbf p'_- $ divisible by $\mathbf  p $ (since $  \bm  \upchi $ and $ \mathbf p $ are relatively prime). Since we have shown that $ \mathbf p'_- $ divides $ \mathbf p_- $, it follows that $ \mathbf q\mathbf p_+\mathbf p'_- $ divides $\mathbf q\mathbf p_+ \mathbf p_-=\mathbf p $. So $ \mathbf q\mathbf p_+\mathbf p'_- $ is a nonzero scalar multiple of $\mathbf  p $, which is not possible  if $ \mathbf p'_- $ is a strict divisor of $ \mathbf p_- $. A similar argument works for $ \mathbf p_+, \mathbf p'_+ $.

Let us consider now again the matrix $ E $ defined by~\eqref{eq:def of E}. Since $ E\in\mathfrak{B}(\mathbf p_+,\mathbf p_-, \bm  \upchi')  $, we have
 $ \mathbf p_+(M) \bm  \upchi(M)=\mathbf p_+(M)\mathbf b'_1(M) $, 
$  \mathbf p_-(M)=\mathbf p_-(M)\mathbf b'_{1-n}(M)  $ for some polynomials $ \mathbf b_1' ,\mathbf  b'_{1-n}  $ such that $\mathbf  q $ divides
$ \mathbf b_1'  - \bm  \upchi' \mathbf b_{1-n}' $. So there are polynomials $\mathbf  u ,\mathbf v ,\mathbf w  $ such that
\begin{align}
 \bm  \upchi &=\mathbf b'_1 +\mathbf p_- \mathbf d \mathbf u ,\label{eq:21}\\
\bm 1&=\mathbf  b'_{1-n} +\mathbf p_- \mathbf d \mathbf v , \label{eq:22}\\
\mathbf b_1'  - \bm  \upchi' \mathbf b_{1-n}' &=\mathbf d \mathbf w \label{eq:23}
\end{align}
Multiplying~\eqref{eq:22} with $  \bm  \upchi'  $, substracting from~\eqref{eq:21}, and using~\eqref{eq:23},  we obtain
\[
 \bm  \upchi - \bm  \upchi' =\mathbf d \mathbf w + \mathbf p_- \mathbf d \mathbf u -\mathbf p_-\mathbf  d  \bm  \upchi'\mathbf  q\mathbf  v ,
\]
whence $\mathbf  d  $ divides $  \bm  \upchi - \bm  \upchi' $.
	\end{proof}

\begin{remark}\label{re:cases}
	(i) Some particular cases of these algebras are worth mentioning. 
	If $ \mathbf p_+=\mathbf p $, $ \mathbf p_-=\bm 1 $, and $\bm  \upchi=\bm 1 $, then condition (b) of Definition~\ref{de:algebras A(r, rho)} says that $ A_i=0 $ for $ i\ge 1 $, and it is not subject to any conditions for  $i\le 1  $. Also, since $ \mathbf q=\bm 1 $, condition (c) does not entail any restriction. So $ \mathfrak{B}(\mathbf p,\mathbf 1, \bm 1) $ is the algebra $ \mathfrak{B}_u $ of upper triangular block Toeplitz matrices with entries in $ \AA $. Similarly, $ \mathfrak{B}(\mathbf 1,\mathbf p, \bm 1) $ is the algebra $ \mathfrak{B}_l $ of lower triangular block Toeplitz matrices with entries in $ \AA $. 
	
	(ii) Another noticeable case is $ \mathbf p_+\mathbf p_-=\mathbf p$. Then condition (i) in Definition~\ref{de:algebras A(r, rho)} as well as (3)(iii) in Theorem~\ref{th:uniqueness} are automatically satisfied, 
	whence $ \mathfrak{B}(\mathbf p_+,\mathbf p_-, \bm 1\upchi) $ does depend on its third argument.
\end{remark}

In the sequel we will call  an algebra $ \BB $ \emph{generic} if it is neither contained in $ \mathfrak{B}_u $ nor in $ \mathfrak{B}_l $.

\section{Classification of algebras in the singly generated case}

The purpose of this section is to show that the algebras defined in the previous section are  all the maximal algebras of block Toeplitz matrices with entries in $ \PP(M) $. 

We start with a linear algebra lemma whose proof is left to the reader.

\begin{lemma}\label{le:general vector lemma}
	Let $E_1,E_2,\cdots E_n$ be the subspaces of vector space $E$ such that for every $k=1,2,\cdots n$, $E_k\neq E$. Then $E\neq\bigcup_{k=1}^{n} E_k$.
\end{lemma}

In the rest of this section $ M $ is again a fixed nonderogatory matrix of order~$ m $, with minimal polynomial $\mathbf  p $. 

The next lemma gives some properties of generic maximal algebras.

\begin{lemma}\label{le:s_j constant} Suppose $ \BB\in \TT_{n,d}[\PP(M)] $ is a generic maximal algebra. Then:
	\begin{itemize}
		\item[\rm (i)] For any value of $ j $ there exists at least an element $ B\in\BB $ with $ \mathbf b_j\not=0 $.
		
		\item[\rm(ii)]   If we define $ \mathbf s_j  $ to be the greatest common divisor of all $ \mathbf b_j  $ with $ B\in\BB $ and $\mathbf  p $, 
		then
		$ \mathbf s_1=\dots=\mathbf s_{n-1} $ and $ \mathbf s_{-1}=\dots=\mathbf s_{-(n-1)} $.
		
		Consequently, if we denote the first polynomial in {\rm (ii)} by $ \mathbf s_+ $ and the second by $ \mathbf s_- $, any element $B\in\mathcal{B}$ can be written $\mathbf b_j(M)=\mathbf s_+(M)\tilde{\mathbf b}_j(M)$ if $ j\ge 1 $ and $\mathbf b_j(M)=\mathbf s_-(M)\tilde{\mathbf b}_j(M)$ if $ j\le -1 $.
		
		\item[\rm (iii)] There exists an element $ A\in\BB $ such that $ \tilde{\mathbf a}_j $ and $\mathbf  p $ are relatively prime for for every $j=\pm 1,\cdots\pm (n-1)$.
		
		\item[\rm (iv)]  $ \mathbf s_+\mathbf s_-$ divides $\mathbf  p$.
	\end{itemize}
	
\end{lemma}

\begin{proof}
	(i) Suppose $ j\ge 1 $ is such that $ \mathbf b_j=0 $ for all $ B\in\BB $. It follows by  Theorem~\ref{th:general facts about maximal algebras} (iii) that $ \mathbf b_i=0 $ for all $ B\in\BB $ and $ i\ge 1 $. Then $ \BB $ is formed only by upper triangular matrices. A similar argument implies that if $ j\le -1 $ is such that $ \mathbf b_j=0 $ for all $ B\in\BB $, then $ \BB $ is formed only by lower triangular matrices.

	(ii) is also an immediate consequence of Theorem~\ref{th:general facts about maximal algebras} (iii). 
	
	(iii) 
	Denote by $\mathbf t^+_1,\dots, \mathbf t^+_{r_+} $ the first degree factors of $\mathbf  p/\mathbf s_+ $ and by $\mathbf t^-_1,\dots, \mathbf t^-_{r_-} $ the first degree factors of $ \mathbf p/\mathbf s_- $. Define the subspaces 
	\[
	\begin{split}
	E_{i,r}&=\{B\in\BB:\mathbf  t^+_r|\tilde{\mathbf b}_i \},\quad
	i=1,\dots, n-1, \quad r=1, \dots, r_+,\\
		E_{i,r}&=\{B\in\BB: \mathbf t^-_r|\tilde{\mathbf b}_i \},\quad
	i=-(n-1),\dots, -1, \quad r=1, \dots, r_-.
	\end{split}
	\]
	Take $ i\ge1 $; then	$ E_{i,r}\not=\BB $, since otherwise  we could take 
	  $ \mathbf s_+ \mathbf t^+_r $ instead of $ \mathbf s_i$ (actually, instead of $ \mathbf s_+ $ by (ii)). Similarly, $ E_{i,r}\not=\BB $ for $ i\le -1 $, since otherwise we could take $ \mathbf s_- \mathbf t^-_r $ instead of $ \mathbf s_-$.  By Lemma~\ref{le:general vector lemma} it follows that there exists $ A\in\BB $ which does not belong to any of $ E_{i,r} $; that $ A $ is the required element.
	
	(iv)
	Take  $A\in \BB $ obtained at point (iii); we may also assume, for simplicity, that $ A\in\CD_1 $. For an arbitrary
	 $ B\in\BB $ arbitrary, using again the notation of (ii), condition~\eqref{prod} implies that
	\begin{equation}\label{eq:division}
	\mathbf p | \mathbf s_+\mathbf s_-(\tilde{\mathbf a}_1\tilde{\mathbf b}_{i-n}-\tilde{\mathbf a}_{1-n}\tilde{\mathbf b}_1)
	\end{equation}
	 for all $ i\ge 1 $ .
	
	 Since $ \mathbf s_+, \mathbf s_- $ are both divisors of $ \mathbf p $, if $ \mathbf s_+\mathbf s_- $ does not divide $ \mathbf p $ it follows that among the divisors of $\mathbf  p $ of degree 1 there exists at least one, say $ \mathbf t $, that is a divisor of $ \mathbf s_+ $ and has strictly larger degree in $ \mathbf s_+\mathbf s_- $ than in $\mathbf  p $. If $ \mathbf s_+=\mathbf t \mathbf s_+' $, then~\eqref{eq:division} implies that  $\mathbf  p $ divides also $ \mathbf s_+'\mathbf s_-(\tilde{\mathbf a}_1\tilde{\mathbf b}_{i-n}-\tilde{\mathbf a}_{1-n}\tilde{\mathbf b}_1) $. 
	 
	 Consider then the Toeplitz cyclic diagonal $ A' $ of order~1 defined by 
	 \[
	  A'_1= \mathbf s_+'(M)\tilde {\mathbf a}_1(M) ,\quad  A'_{1-n}=\mathbf s_-(M) \tilde{\mathbf a}_{1-n}(M) .
	 \]
	  It is not in $ \BB $ since $\mathbf  s'_+ $ is a strict divisor of $ \mathbf s_+ $ and $ \tilde{\mathbf a}_1 $ is relatively prime with $\mathbf  p $,	
	but it satisfies~\eqref{prod} for all $ B\in \BB $. 
	By Corollary~\ref{co:adding a cyclic diagonal} the algebra generated by $ \BB $ and $ A $ is formed by block Toeplitz matrices, and it would contain $ \BB $ strictly---a contradiction. It follows therefore that $ \mathbf s_+\mathbf s_- $ divides $\mathbf p$.
\end{proof}

We will then denote $ \mathbf d=\mathbf p/\mathbf s_+\mathbf s_- $.
Fixing $ A $ as in Lemma~\ref{le:s_j constant} (iii), take $ i\ge1 $.
Since $ \tilde{\mathbf a}_{i-n} $ is relatively prime to $\mathbf  p $, we may choose 
 $ \bm  \upgamma_i\in\mathbb{C}[X] $ such that  $\mathbf  p $ divides $ \bm  \upgamma_i \tilde{\mathbf a}_{i-n} -1$, and define 
\[
\bm  \upxi_i =\tilde{\mathbf a}_i \bm   \upgamma_i .
\]

\begin{lemma}\label{le:preparatory for reciprocal}
	With the above notations,   $ \mathbf  d $ divides
	 $  \bm  \upxi_i - \bm  \upxi_j  $ for any $ i,j\ge 1$.
\end{lemma}

\begin{proof}
Theorem~\ref{th:general facts about maximal algebras} (ii) implies that $ \mathbf{E}_k(A)\in \mathcal{B} $ for all $ k\ge 1 $. Therefore $ \mathbf{E}_i(A) \mathbf{E}_j(A)\in \mathcal{B} $ is a block Toeplitz matrix and satisfies~\eqref{prod}.
	Thus
	\begin{equation*}\label{eq:prep1}
	\mathbf s_+(M)\mathbf s_{-}(M)\tilde{\mathbf a}_i(M)\tilde{\mathbf a}_{j-n}(M)=
\mathbf 	s_+(M)\mathbf s_{-}(M)\tilde{\mathbf a}_j(M)\tilde{\mathbf a}_{i-n}(M),
	\end{equation*}
	whence 
\[
\tilde{\mathbf a}_k \tilde{\mathbf a}_{j-n} 
-\tilde{\mathbf a}_j \tilde{\mathbf a}_{k-n} = \mathbf d \mathbf u 
\]
for some polynomial $\mathbf  u $. Multiplying with $\bm  \upgamma_i \bm  \upgamma_j  $, we have 
\[
\bm  \upgamma_i \tilde{\mathbf a}_i\bm   \upgamma_j \tilde{\mathbf a}_{j-n} 
-\bm  \upgamma_j \tilde{\mathbf a}_j \bm  \upgamma_i \tilde{\mathbf a}_{i-n} 
=\bm  \upgamma_i \bm  \upgamma_j \mathbf d\mathbf  u .
\]
The definitions of $\bm   \upgamma_i,\bm   \upgamma_j $ and $  \bm  \upxi_i,  \bm  \upxi_j $	imply then that 
\[
 \bm  \upxi_i - \bm  \upxi_j = \mathbf d \mathbf  v 
\]
for some polynomial $ \mathbf v  $.
\end{proof}

Define then $  \bm  \upxi = \bm  \upxi_1  $. We get now  close to the main result of this section.

\begin{corollary}\label{co:main converse}
	Suppose that $ \mathcal{B} $ is a maximal generic algebra in $\mathcal{T}_{n,d}$, with entries in $ \mathcal{P}(M) $. If $ \mathbf s_+, \mathbf s_-,  \bm  \upxi $ are defined as above, then $ \mathcal{B}=\mathfrak{B}({\mathbf s_+,\mathbf s_-, \bm  \upxi}) $.
\end{corollary}

\begin{proof}
	It is enough to prove that $ \mathcal{B}\subset \mathfrak{B}({\mathbf s_+,\mathbf s_-, \bm  \upxi})  $, since  maximality implies  that  the inclusion is actually an equality.
		
	Take $ B\in \mathcal{B} $.	
Since $ AB\in\mathcal{B} $, it follows by Lemma~\ref{le:basic condition for product toeplitz} that
	\[
	\mathbf s_+(M)\mathbf s_-(M)\tilde{\mathbf a}_i(M)\tilde{\mathbf b}_{j-n}(M)=\mathbf s_+(M)\mathbf s_-(M)\
\tilde{\mathbf a}_{i-n}(M)\tilde{\mathbf b}_{j}(M),
	\]
	which is equivalent to 
	\[
	\tilde{\mathbf a}_i \tilde{\mathbf b}_{j-n} -
	\tilde{\mathbf a}_{i-n} \tilde{\mathbf b}_{j} =\mathbf d\mathbf  u 
	\]
	for some polynomial $\mathbf  u  $.
	Multiplying with $ \bm \upgamma_i  $, we obtain the  relation
	\[
 \bm  \upxi_i \tilde{\mathbf b}_{j-n} -	\tilde{\mathbf b}_j =\mathbf d \mathbf v 
	\]	
	for some polynomial $\mathbf  v  $. A final application of Lemma~\ref{le:preparatory for reciprocal} implies that we can replace in the last equality $  \bm  \upxi_i $ by $  \bm  \upxi_1= \bm  \upxi $. Therefore $ B\in \mathfrak{B}({\mathbf s_+,\mathbf s_-, \bm  \upxi})  $, which finishes the proof.
\end{proof}

We can now finish the classification by removing the genericity condition.

\begin{theorem}\label{th:general classification}
	A maximal algebra $ \BB $ of block Toeplitz matrices with entries in $ \PP(M) $ is of type $ \mathfrak{B}({\mathbf p_+,\mathbf p_-, \bm  \upchi})  $ for some polynomials $ \mathbf p_+, \mathbf p_-,  \bm  \upchi  $, with $ \mathbf p_+\mathbf p_-$ dividing $\mathbf p $ and $  \bm  \upchi, \mathbf p  $ relatively prime.
\end{theorem}

\begin{proof}
	If $ \BB $ is generic, we know from Corollary~\ref{co:main converse} that the statement is true by taking $ \mathbf p_+=\mathbf s_+ $, $ \mathbf p_-=\mathbf s_- $, $ \bm\upchi=\bm\upxi $.
	 If  $ \BB $ is not generic, then, by definition, it must be  contained either in 
	 $\mathfrak{B}_u = \mathfrak{B}(\mathbf p,\mathbf 1, \bm 1)$ or in $ \mathfrak{B}_l =\mathfrak{B}(\mathbf 1,\mathbf p, \bm 1) $. But these  are both maximal algebras, so $ \BB $ has to be one of them.
	\end{proof}

\begin{example}\label{ex:k+k-}
	Suppose $ \mathbf{p}(X)=X^m $ for some positive integer $ m $. The divisors of~$ \mathbf{p} $ are nonzero multiples of $ X^k $ for $ 0\le k\le m $, while a polynomial is prime with $ X^m $ precisely when it has nonzero free term. We leave then to the reader to check that the above classification yields  a one-to-one correspondence between maximal algebras and the union $ \bm T\cup\bm P $, where:
	
	\begin{itemize}
		\item $ \bm T $ is the set of triples $ (k_+, k_-, \bm \upchi) $, with $ k_+, k_- $  nonnegative integers such that $ k_++k_-< m $, while $ \bm \upchi $ is a polynomial of degree less or equal to $m-k_+-k_- -1 $ with nonzero free term. The corresponding maximal algebra is $ \mathfrak{B}(X^{k_+} ,X^{k_-},  \bm \upchi) $.
		
		\item
		$ \bm P $ is the set of pairs $ (k_+, k_-) $, with $ k_+, k_- $  nonnegative integers with $ k_++k_-= m $.  The corresponding maximal algebra is $ \mathfrak{B}(X^{k_+} ,X^{k_-},  \bm 1) $.
	\end{itemize}

	In the second case (which corresponds to Remark~\ref{re:cases} (ii)) it is easy to see that for $ k_+, k_-<m $ the algebra $ \mathfrak{B}(X^{k_+} ,X^{k_-},  \bm 1) $ is not a pseudocirculant algebra. Indeed, assuming this true, let  $ Y=(Y_{i-j})_{i,j=0}^{n-1} $ be defined by $ Y_i=M^{k_+} $ for $ i\ge 0 $ and $ Y_i=0 $ for $ i<0 $. Then $ Y\in \mathfrak{B}(X^{k_+} ,X^{k_-},  \bm 1) $, so we should have $ AM^{k_+}=0 $. Therefore $ A=a(M) $ is not invertible. Since invertible elements in $ \PP(M) $ are characterized by nonzero free term, it follows that $ a(M) $ has $ M $ as a factor. 
	
	Similarly, by considering $ Y'=(Y'_{i-j})_{i,j=0}^{n-1} $,  defined by $ Y'_i=M^{k_-} $ for $ i\le 0 $ and $ Y'_i=0 $ for $ i>0 $, one obtains that $ B=b(M) $ is not invertible, and thus $ b(M) $ has also $ M $ as a factor. Therefore $ \ker A\cap\ker B\supset \ker M $, and condition~\eqref{eq:kerAkerB} cannot be satisfied. 
	
\end{example}


\end{document}